\documentclass[11pt,a4paper]{amsart}
\usepackage{amssymb,amsmath,amsthm}

\newtheorem*{theorem}{Theorem}

\newtheorem{lemma}{Lemma}
\newtheorem{conj}{Conjecture}
\newtheorem{prop}{Proposition}

\newtheorem{defi}[]{Definition}

\renewcommand{\geq}{\geqslant}
\renewcommand{\leq}{\leqslant}
\renewcommand{\mod}{\, \textrm{mod }}

\newcommand{\C}{\mathbb{C}}
\newcommand{\R}{\mathbb{R}}
\newcommand{\eps}{\varepsilon}

\numberwithin{equation}{section}

\begin{document}

\title{Chebyshev constants for the unit circle}

\author{Gergely Ambrus}
\email[G. Ambrus]{ambrus@renyi.hu}
\address{Alfr\'ed R\'enyi Institute of Mathematics, Hungarian Academy of Sciences, PO Box 127, 1364 Budapest, Hungary}
\thanks{The research of the first named author was supported by OTKA grants 75016 and
76099, and by the UCL Graduate School Research Scholarship.}

\author{Keith M. Ball}
\email[K. Ball]{kmb@math.ucl.ac.uk}
\address{Department of Mathematics, University College London, Gower Street, London WC1E 6BT, United Kingdom}

\author{Tam\'as Erd\'elyi}
\email[T. Erd\'elyi]{erdelyi@math.tamu.edu}
\address{Department of Mathematics, Texas A\&M University, College Station, Texas 77843, U.S.A}

\begin{abstract}
It is proven that for any system of $n$ points $z_1, \dots, z_n$ on the
(complex) unit circle, there exists another point $z$ of norm 1, such that
\[
\sum \frac{1}{|z_k-z|^2} \leq \frac{n^2}{4}.
\]
Two proofs are presented: one uses a characterisation of equioscillating
rational functions, while the other is based on Bernstein's inequality.

\end{abstract}


\maketitle

\section{Introduction and results}\label{introd}

Chebyshev constants were defined by Fekete \cite{Fe23} and P\'olya and
Szeg\H o~\cite{PSz}. Since then the notion became  fundamental in classical
potential theory. There are several different definitions available, among
which the following is the most suitable for our purposes.

\begin{defi}
Let $K$ be a compact set in a normed space $(X,\|.\|)$. For a given point set
$(x_k)_1^m$ on $K$, let
\[
M^p(x_1, \dots, x_n) = \min_{x \in K} \sum_{k=1}^n \frac{1}{\|x-x_k\|^p}
\]
for $p > 0$, and
\[
M^0(x_1, \dots, x_n) = \min_{x \in K} \prod_{k=1}^n \frac{1}{\|x-x_k\|}\, .
\]
The $n$th $L_p$ Chebyshev constant of $K$ is then given by
\[
M_n^p(K) = \max_{x_1, \dots, x_n \in K} M^p(x_1, \dots, x_n).
\]
\end{defi}
Note that $M_n^0(K)$ is the reciprocal of the usual (modified) $n$th Chebyshev
constant of $K$, cf. page 39 of \cite{BE}. The sum $\sum \|x-x_k\|^{-p}$ is the
Riesz potential of order $(2-p)$ of the discrete distribution with the $x_j$'s
as atoms of weight 1, see  \cite{La}.

Throughout the article, the role of $K$ will be played by the complex unit
circle $T$ endowed with the natural norm. It is natural to expect that in this
situation, the point sets maximising $M^p(z_1, \dots, z_n)$ are equally
distributed on the circle. For $M_n^0(T)$, this assertion is well known, and
easy to obtain by applying the method for proving  the extremality of Chebyshev
polynomials. We illustrate a proof that is parallel to the subsequent
arguments; note that this proof is not unique, see e.g. Theorem 2 of
\cite{ARe}.

\begin{prop}
For any positive integer $n$, $M_n^0(T)=1/2$.
\end{prop}
\begin{proof}
First, we show that $M_n^0(T) \leq 1/2$. Assume on the contrary that there
exists a set $(z_k)_1^n$ of complex numbers of norm 1, such that $|q(z)|<2$ for
every $z \in T$ with
\begin{equation} \label{qz}
q(z) = \prod_{k=1}^n (z-z_k).
\end{equation}
Note that for $v,w \in T$,
\begin{equation}\label{vw}
(v-w)^2 =    - v \,w\,\varrho
\end{equation}
holds with some $\varrho \geq 0$. Thus, $q^2(z)$ can be written as
\[
q^2(z)= (-1)^n \gamma \, z^n \varrho(z),
\]
where $\gamma = \prod z_k$, and $\varrho(z)$ is a real function defined on $T$,
taking only non-negative values. By rotation, we may assume that $\gamma =
(-1)^{n-1}$, and thus
\begin{equation}\label{q2z}
q^2(z) =- z^n \varrho(z).
\end{equation}
 According to the assumption, $0 \leq \varrho(z)<4$ for
all $z \in T$. Take now
\begin{equation}\label{Qz}
Q(z) =  (z^n - 1)^2 =  -z^n R(z).
\end{equation}
Here $R(z)$ is again a real function on $T$ satisfying $0 \leq R(z) \leq 4$;
moreover, $R(z)=0$ or $R(z)=4$ at exactly $2n$ points on $T$. Thus, counting
with multiplicities, the function $\varrho(z) - R(z)$ has at least $2n$ zeroes
on $T$, and by \eqref{q2z} and \eqref{Qz}, the same holds for $q^2(z)-Q(z)$.
However, this contradicts the fact that  $q^2(z)-Q(z)$ is a polynomial of
degree at most $2n-1$.

To see that $M_n^0(T) \geq 1/2$, choose the $n$th unit roots: $z_k = \xi_k =
e^{i 2 \pi k /n }$.
\end{proof}

Prior to the present article, the exact value of $M^p_n(T)$ was not known
except for $p=0$ or for small values of $n$. The analogous problem for $p<0$
was considered and partly solved by Stolarsky~\cite{Sto}. For $p>0$, the
following observation gives an upper estimate of $M^p_n(T)$ for any $n \geq 1$.

\begin{prop}\label{est_erdelyi}
The $L_p$ Chebyshev constants of $T$ can be estimated as
\begin{displaymath}
M^p_n(T) < \left\{ \begin{array}{ll}
c_0 \,  n /(1-p) & \textrm{for $0 < p <1$}\\
c_1 \, n (1 + \ln n) & \textrm{for $p=1$}\\
c_p \,n^p & \textrm{for $p>1$}
\end{array} \right.
\end{displaymath}
with positive constants $c_0$ and $c_p\, ( p\geq 1)$.
\end{prop}

\begin{proof}
Let $q(z)$ be the polynomial defined by \eqref{qz}, and let $z_0$ be a point on
$T$ where $q(z)$ attains its maximal modulus on $T$. According to exercise E.12
on page 237 of \cite{BE}, for every $r>0$, there are at most $e n r$ zeroes of
$q(z)$ on the arc $\{ z_0\, e^{i \rho}: \, \rho \in [-r,r] \}$. Thus,
\begin{equation}\label{est}
\sum_{k=1}^n \frac{1}{|z_0 - z_k|^p} < \sum_{k=1}^n  \left( 2
\sin\left(\frac{k}{2 en} \right)\right)^{-p} < (3\,e)^p n^p \sum_{k=1}^n \left(
\frac {1}{k} \right) ^p,
\end{equation}
where at the last inequality the constant 3 can be replaced by $1 / (2 e
\sin(1/2 e))$.

 For $p>1$, we obtain that
\begin{equation}\label{est_p>1}
M^p_n(T)< (3 e)^p \, \zeta(p) \, n^p.
\end{equation}
For $p=1$, \eqref{est} yields
\[
M^1_n(T)< 3 e \, n (1 + \ln n),
\]
whereas for $0 <p <1$,
\[
M^p_n(T)< (3 e)^p \left( \frac {n}{1-p} - \frac {p}{1-p}\right). \qedhere
\]
\end{proof}

Taking the system of the $n$th unit roots, $\xi_k = e^{i 2 \pi k/n},\ k=1,
\dots, n$, shows that the estimate provided by Proposition~\ref{est_erdelyi} is
asymptotically sharp in terms of $n$. The complete asymptotic expansion of
$M^p(\xi_1, \dots, \xi_n)$ was established by Brauchart, Hardin and Saff
\cite{BHS}. More precisely, they determined the asymptotic expansion of the
{\em Riesz $p$-energy}
\[
E^p(\xi_1, \dots, \xi_n) = \sum _{j=1}^n \sum_{\begin{subarray}{l} k=1\\
k \neq j
\end{subarray}}^n
\frac{1}{|\xi_j-\xi_k|^p}\, .
\]
Using the notation $E^p_n = E^p(\xi_1, \dots, \xi_n)$, the expansion of
$M^p(\xi_1, \dots, \xi_n)$ can  be obtained by the formula
\[
M^p(\xi_1,\dots, \xi_n) = \frac{E^p_{2n}}{2n} - \frac{E^p_n}{n}\,.
\]
In comparison with Proposition~\ref{est_erdelyi}, we note that the results of
\cite{BHS} yield that
\begin{align*}
M^p(\xi_1,\dots, \xi_n)  &\approx (2^p-1) \zeta(p) n^p  &\textrm{for $p>1$,}\\
M^p(\xi_1,\dots, \xi_n)  &\approx \frac 1 \pi \, n \ln n  &\textrm{for
$p=1$, and}\\
M^p(\xi_1,\dots, \xi_n)  &\approx \frac {2^{-p}}{\sqrt{\pi}} \frac
{\Gamma((1-p)/2)}{\Gamma(1-p/2)}\, n &\textrm{for $0<p<1$.}
\end{align*}

That the point systems on $T$ minimising the Riesz energies are equally
distributed on $T$ was (essentially) proved by Fejes T\'oth \cite{FT}. The case
of Chebyshev constants is much harder to tackle. One may even conjecture that
the equally distributed case is extremal in a more general setting.

Let $f$ be an even, $2 \pi$-periodic real function. We say that $f$ is {\em
convex}, if it is convex on $(0, 2 \pi)$; in this case, $f$ has a maximum at
$0$. We allow $f$ to have a pole at $0$. For $\theta_1, \dots, \theta_n \in [0,
2\pi)$, let
\[
S_f(\theta) = \sum_{k=1}^n f(\theta - \theta_k),
\]
and
\[
M_f(\theta_1, \dots, \theta_n) = \min_{\theta \in [0,2 \pi)} S_f(\theta).
\]

 We  say that $\theta_1, \dots, \theta_n \in [0, 2\pi)$ is {\em
uniformly distributed on $[0,2 \pi)$}, if for some $\lambda \in \R$,
\[
\{\theta_1, \dots, \theta_n\} = \left \{\left(\lambda +  \frac {2 \pi k }
{n}\right) \mod 2 \pi: k= 1, \dots, n \right \}.
\]

\begin{conj}\label{convconj}
For any $2 \pi$-periodic, even, convex function $f$, $M_f(\theta_1, \dots,
\theta_n)$ is maximised when $(\theta_k)_1^n$ is uniformly distributed on
$[0,2\pi)$.
\end{conj}

It is easy to see that  for $p>0$, the function $f(\theta) = \sin ^{-p}
(\theta/2)$ satisfies the above conditions. Thus, in particular, we conjecture
that for any $p > 0$ and $n \geq 1$,
\begin{equation}\label{mptextr}
M^p_n(T) = M^p(\xi_1,\dots, \xi_n).
\end{equation}

A strong indication for the validity of Conjecture~\ref{convconj} is the
following fact. A set of points $(\theta_j)_1^n \subset [0, 2\pi)$ is {\em
locally maximal with respect to $f$}, if there exists  $\nu>0$, such that for
any $(\theta'_j)_1^n \subset [0, 2\pi)$ satisfying $|\theta_j - \theta'_j| <
\nu \, (\mod 2\pi)$ for every $j$,
\[
M_f(\theta_1, \dots, \theta_n) \geq M_f(\theta'_1, \dots, \theta'_n).
\]
Note that when $f$ is strictly convex, locally minimal cases are of little
interest: then $\theta_1= \dots = \theta_n$. Clearly, $S_f$ is convex on the
intervals between consecutive $\theta_j$'s, and if $f$ is strictly convex, then
$S_f$ has exactly one local minimum on each of these intervals.

\begin{lemma}\label{equimini}
If $f$ satisfies the above conditions, and $(\theta_j)_1^n$ is a locally
maximal set with respect to $f$, then all the local minima of $S_f(\theta)$ are
equal.
\end{lemma}

\begin{proof}
Let $\Theta = (\theta_j)_1^n$, and assume on the contrary that
$0<\theta_1<\theta_2<2 \pi$ are two consecutive points such that the local
minimum of $S_f(\theta)$ on the interval $[\theta_1, \theta_2]$ is strictly
larger than $M_f(\Theta)$. Let $\eps$ be a small positive number, and consider
the new set of points $\Theta'$ obtained from $\Theta$ by exchanging $\theta_1$
and $\theta_2$ for
\[
\theta_1' = \theta_1- \eps \; , \; \theta_2' = \theta_2+\eps.
\]
Let $S'_f(\theta)$ be the function determined by the point set $\Theta'$. By
the symmetry and convexity of $f$, it is easy to verify that for  $\theta_1
\leq \theta \leq \theta_2$,
\begin{equation}\label{sfchange}
S_f(\theta) \geq S'_f(\theta).
\end{equation}
Interchanging the role of $\theta_1$ and $\theta_2$, it also follows that for
$\theta \in [0,2\pi] \setminus [\theta_1', \theta_2']$, the reverse of
\eqref{sfchange} holds. If $\eps>0$ is sufficiently small, then the global
minimum of $S'_f(\theta)$ is still attained on $[0,2\pi] \setminus [\theta_1',
\theta_2']$. Thus, $M_f(\Theta) < M_f(\Theta')$, which contradicts the
extremality of $(\theta_j)_1^n$.
\end{proof}

Analogous equioscillation properties often arise in the context of minimax
problems. For instance, Bernstein conjectured that for the optimal set of nodes
for Lagrange interpolation,  the so-called Lebesgue function of the projection
is equioscillating. This long-standing problem was solved by Kilgore
\cite{K76}. In this case, the equioscillation property also leads to the
characterisation of the optimal  nodes, see \cite{K78} and \cite{dBP}.

 The main goal of the paper is to give a complex analytic proof for the $p=2$
case of \eqref{mptextr}.

\begin{theorem}
For any set $z_1, \dots, z_n$ of complex numbers of modulus 1, there exists a
complex number $z_0$ of norm 1, such that
\begin{equation}\label{planpolineq}
\sum_{j=1}^n \frac{1}{|z_0-z_j|^2} \leq \frac{n^2}{4}\;.
\end{equation}
The inequality is sharp if and only if the numbers $z_j$ are distinct and there
exists a $c \in T$ such that $z_j^n = c$ for all $j = 1,2,\dots,n$.
\end{theorem}

We present two proofs for the above Theorem. The first one is based on
Lemma~\ref{equimini} and it uses the equioscillating property of the functions
arising in local extremal cases. The second proof refers to Bernstein's
inequality, which in turn is based again on the equioscillating property. We
believe that both of these proofs are of independent interest.

We note that the problem arose in connection with the so-called polarisation
problems. In fact, the Theorem is the planar case of the strong polarisation
problem. There is a third proof of it following these lines; we refer the
interested reader to \cite{A}. This proof uses directly the extremal property
of the Chebyshev polynomials, and thus it is a close relative of the second
proof presented here. For reasons of conserving space, we do not give a
detailed description of the polarisation problems, which the interested reader
can find in \cite{A} and in \cite{PR}. Also, numerous results in potential
theory are related to the present problem, see e.g. \cite{FN}.

\section{Complex analytic tools}\label{proofs}

The following notion will play an important role; we follow Szeg\H o~\cite{Sze}.

\begin{defi}\label{inverspol}
Let $g(z)= a_0 + a_1 z + \dots + a_n z^n$ be a complex polynomial. Its {\em
reciprocal polynomial of order $n$} is defined by
\[
g^*(z)=  \bar{a}_n + \bar{a}_{n-1}z + \dots +\bar{a}_0 z^n.
\]
\end{defi}

If we do not specify otherwise, $g^*(z)$ will denote the reciprocal polynomial
whose order is the exact degree of $g$.  For any non-zero complex number $z$,
let $z^*$ denote its image under the inversion with respect to complex unit
circle $T$: $ z^* = 1/\bar z$. Clearly,
\begin{equation}\label{zz*}
g^*(z) = z^n \overline{g(z^*)},
\end{equation}
and thus, if the non-zero roots of $g(z)$ are $\zeta_1, \dots , \zeta_n$, then
the non-zero roots of $g^*(z)$ are $\zeta^*_1, \dots , \zeta^*_n$.

In particular, if all zeroes of $g(z)$ lie on $T$, then the zeroes of $g(z)$
and $g^*(z)$ agree, hence we obtain the following.

\begin{prop}\label{gg*}
If all zeroes of the polynomial $g(z)$ have modulus~1, then
\[
g^*(z) = \gamma g(z)
\] for a complex constant $\gamma$ with $|\gamma|=1$.
\end{prop}

Lemma~\ref{equimini} implies that for a stationary point set, the resulting
function $S_f(\theta)$ is equioscillating. In the special case that we treat,
$S_f(\theta)$ can be written as the real part of a complex rational function.
 In light of these observations, we introduce the following concept.

\begin{defi}\label{equioscill}
The real valued function $f$  on $T$ is {\em equioscillating of order $n$}, if
there are $2n$ points $w_1, w_2, \dots, w_{2n}$ on $T$ in this order, such that
\[
f(w_j) = (-1)^{j} \|f\|_T
\]
for every $j=1, \dots , 2n$, and $|f(z)|<\|f\|_T$ if $z \neq w_j$ for any $j$.
\end{defi}

Although equioscillation in general is not a very specific property (plainly,
any real valued function on $T$ whose level sets are finite has a shifted copy
which is equioscillating of some order), equioscillation of a possible maximal
order is a strong condition. This becomes apparent in the context of rational
functions.

Suppose that $R(z)$ is a rational function, whose numerator is of degree $k$
and whose denominator has degree $l$; then the real and imaginary parts of
$R(z)$ are the quotients of two trigonometric polynomials of degrees $k$ and
$l$, and therefore $\Re (R(z))$ and $\Im( R(z))$ cannot be equioscillating of
order larger than $\max \{k,l \}$. These simple observations are intimately
tied to the right Bernstein-type inequalities for spaces of rational functions
on the unit circle as well as for spaces of ratios of trigonometric polynomials
on the period, see \cite{BE96} and \cite{BE94}, respectively.

A characterisation of those rational functions whose real and imaginary parts
are oscillating with the maximal possible order was given by Glader and H\"ogn\"{a}s
\cite{GH}. They showed that if $R(z)$ is a rational function with numerator and
denominator degrees at most $n$, and $\Re (R(z))$ and $\Im( R(z))$ are
equioscillating functions on $T$ of order $n$, then
\[R(z) = c \,B(z) \textrm{ or } R(z) = c / B(z),
\] where $c$ is a real constant and $B(z)$ is a finite Blaschke product of order
$n$:
\[
B(z)= \rho \,z^k \prod_{j=1}^{n-k} \frac{z- \alpha_j}{1- \bar{\alpha}_j z}\;.
\]
Here $\rho,\alpha_1, \dots, \alpha_{n-k}$ are complex numbers with
$|\rho|=1$ and $0<|\alpha_j|<1$.

The essence of the above result of \cite{GH} is the following statement, for
which we present a simple proof.

\begin{lemma}\label{equifunc}
Suppose that $ w_1, w_2, \dots, w_{2n},w$ are different points on $T$ in this
order. There exists a complex polynomial $h(z)$ of degree $n$, such that
\[
\frac{h(w_k)}{h^*(w_k)}=(-1)^{k+1}
\]
for each $k = 1 ,\dots, 2n$, and
\[
\frac{h(w)}{h^*(w)} = i.
\]
\end{lemma}

\begin{proof}
Taking $g_1(z)= h(z)+h^*(z)$ and $g_2(z)= h(z)-h^*(z)$, the original problem is
equivalent to finding polynomials $g_1(z)$ and $g_2(z)$ of degree $n$ with the
following properties:
\begin{itemize}
\item[(i)] The zeros of $g_1$ are $(w_{2k})$, where $1 \leq k \leq n$;
\item[(ii)] The zeros of $g_2$ are $(w_{2k-1})$, where $1 \leq k \leq n$;
\item[(iii)] $g_1(z) = g^*_1(z) $
\item[(iv)] $g_2(z) = - g^*_2(z) $
\item[(v)] $g_1(w) + i \, g_2(w)=0$.
\end{itemize}
Property (i) is fulfilled, if $g_1(z)$ has the form
\begin{equation}\label{g1}
g_1(z) =  \alpha \prod_{k=1}^n{(z-w_{2k})},
\end{equation}
where $\alpha$ is a complex number of modulus 1. Proposition~\ref{gg*} implies
that property (iii) is satisfied if the leading coefficient and the constant
term of $g_1(z)$ are conjugates of each other, that is,
\[
\bar \alpha = \alpha (-1)^n \prod w_{2k}.
\]
This is achieved by choosing $\alpha$ such that
\begin{equation*}\label{alphaeq}
\alpha^2 =(-1)^n \prod \overline {w}_{2k}.
\end{equation*}
\noindent
 Similarly, conditions (ii) and (iv) are fulfilled if $g_2(z)$ is defined
by
\[
g_2(z) = c \beta \prod_{k=1}^n{(z-w_{2k-1})},
\]
where $c$ is a non-zero real and $\beta$ is a complex number with $|\beta|=1$
satisfying
\begin{equation*}\label{betaeq}
\beta^2 =(-1)^{n+1} \prod \overline {w}_{2k-1}.
\end{equation*}
For any $z \in T$, by (iii), (iv) and the fact $z^*=z$, \eqref{zz*} implies
that
\[
g_1(z) = z^n \overline{g_1(z)}
\]
and
\[
g_2(z) = - z^n \overline{g_2(z)}.
\]
Thus, if neither $g_1$ nor $g_2$ has a zero at $z$, then
\begin{equation}\label{argument}
\arg g_1(z) \equiv \arg (i\, g_2(z)) \equiv \frac n 2 \arg z \pmod \pi.
\end{equation}
In particular, choosing $z = w$, we obtain that the non-zero real $c$ can be
specified so that property (v) holds.
\end{proof}

Finally, we present the variant of Bernstein's inequality that is needed for
the second proof.

An entire function $f$ is said to be of {\em exponential type $\tau$} if for
any $\varepsilon > 0$ there exists a constant $k(\varepsilon)$  such that $|f
(z)| \leq k(\varepsilon)e^{(\tau + \varepsilon)|z|}$ for all $z \in \C$. The
following inequality \cite{Be}, p. 102, is known as Bernstein's inequality. It
can be viewed as an extention of Bernstein's (trigonometric) polynomial
inequality \cite{BE}, p. 232, to entire functions of exponential type bounded
on the real axis.

\begin{lemma}[Bernstein's inequality] Let $f$ be an entire function of
exponential type $\tau > 0$ bounded on ${\Bbb R}$. Then
$$\sup_{t \in \R}{|f'(t)|} \leq \tau \sup_{t \in \R}{|f(t)|}\,.$$
\end{lemma}

The reader may find another proof of the above Bernstein's inequality in
\cite{RS}, pp. 512--514, where it is also shown that an entire function $f$ of
exponential type $\tau$ satisfying
\[
|f'(t_0)| = \tau \sup_{t \in {\R}} |f(t)|
\]
at some point $t_0 \in \R$ is of the form
\begin{equation}\label{e2}
f(z) = ae^{i\tau z} + be^{-i\tau z}\,, \quad a \in \C,  b \in \C, \quad |a| +
|b| = \sup_{t \in {\Bbb R}}{|f(t)|}\,.
\end{equation}

\section{First approach - Equioscillating functions}

\begin{proof}[First proof of the Theorem]
We may assume that $(z_j)_1^n$ is a locally maximal set, and hence it consists
of $n$ different points. Setting
\[ m = 2 \sqrt{M^2(z_1, \dots, z_n)}\,,
\]
the inequality (\ref{planpolineq}), that we wish to prove,  is equivalent to
 $m \leq n$.

 Using that for any $z$ and $z_j$ on $T$,
\begin{equation*}\label{abspol}
|z - z_j|^2 =  - \frac{ (z- z_j)^2}{z \, z_j},
\end{equation*}
we obtain that
\[
\sum_{j=1}^n \frac{1}{|z-z_j|^2} = R^{-1}(z),
\]
where $R(z)$  is the rational function given by
\begin{equation}\label{rationalform}
R(z)  = \frac{\prod_{j=1}^{n}(z-z_j)^2}{-z\sum_{j=1}^{n} z_j \prod_{k \neq j}
(z -z_k)^2}.
\end{equation}

The degrees of the numerator and the denominator of $R(z)$ are $2n$ and at most
$2n - 1$, respectively. The zeroes are $(z_j)_1^n$ with multiplicity 2, and
$R(z)$ assigns real values on the unit circle. Moreover, Lemma~\ref{equimini}
implies that the function
\[
R(z)-\frac{2}{m^2},
\]
which is a rational function as well, oscillates equally between $-2/m^2$ and
$2/m^2$ of order $n$. Let $w_1, \dots, w_{2n}$ be the equioscillation points
such that $w_{2k}= z_k$ for every $k= 1, \dots, n$, and let $w$ be a further
point on $T$ satisfying $R(w) = 2/m^2$. Applying Lemma~\ref{equifunc} yields a
polynomial $h(z)$ of degree $n$, such that
\begin{equation}\label{rateq}
R(z)-\frac{2}{m^2} = \frac{2}{m^2}\, \Re\left(\frac{h(z)}{h^*(z)}\right)
\end{equation}
for every $z= w_1, \dots, w_{2n}, w$. Moreover, both functions assign real
values on $T$, and they have local extrema at the points $(w_j)_1^{2n}$,
therefore their derivatives vanish at these places.

Since $|h(z)|=|h^*(z)|$ on the unit circle,
\[
\frac{2}{m^2}+\frac{2}{m^2}\, \Re\left(\frac{h(z)}{h^*(z)}\right)=
\frac{1}{m^2}\left(2+\frac{h(z)}{h^*(z)}+\frac{h^*(z)}{h(z)}\right)=
\frac{(h(z)+h^*(z))^2}{m^2 \, h(z) h^*(z)}.
\]
Thus, from (\ref{rateq}) we deduce that the rational function
\[
R(z)-\frac{(h(z)+h^*(z))^2}{m^2 \, h(z) h^*(z)}
\]
has double zeroes at all the points $w_1, \dots, w_{2n}$, and it also vanishes
at $w$. On the other hand, its numerator is of degree at most $4n$. Hence, it
must be identically 0, and using~(\ref{rationalform}), we obtain that
\begin{equation}\label{rateq2}
\frac{\prod_{j=1}^{n}(z-z_j)^2}{-z\sum_{j=1}^{n} z_j \prod_{k \neq j} (z
-z_k)^2}=\frac{(h(z)+h^*(z))^2}{m^2 \, h(z) h^*(z)}.
\end{equation}
This equation is the crux of the proof.

 As before, let $g_1(z)= h(z)+h^*(z)$ and $g_2(z)= h(z)-h^*(z)$. Then
by~ \eqref{g1},
\begin{equation}\label{g1z}
g_1(z) =  \alpha \prod_{j=1}^n{(z-z_{j})},
\end{equation}
with a complex number $\alpha$ of norm 1. According to properties (iii) and
(iv),
\begin{equation}\label{g1g2}
\begin{aligned}
g_1(z)&= \alpha z^n + \dots + \bar \alpha, \\
g_2(z)&= \beta  z^n + \dots - \bar \beta,
\end{aligned}
\end{equation}
where now $\beta \in \C \setminus \{ 0 \}$. Substituting $g_1(z)$ and $g_2(z)$,
equation (\ref{rateq2}) transforms to
\begin{equation}\label{rateq3}
\frac{\prod_{j=1}^{n}(z-z_j)^2}{-z\sum_{j=1}^{n} z_j \prod_{k \neq j} (z
-z_k)^2} = \frac{g_1(z)^2}{\frac{m^2}{4} (g_1(z)^2-g_2(z)^2)}\,.
\end{equation}
Since the degree of the denominator on the left hand side is at most $2n-1$,
(\ref{g1g2}) implies that
\begin{equation}\label{leadco}
\alpha = \pm \beta.
\end{equation}
The quotient of the leading coefficients of the numerators on the two sides of
(\ref{rateq3}), which is $\alpha^2$, is the same as the quotient of those of
the denominators. Therefore
\begin{equation*}\label{rateq4}
-\alpha^2 z\sum_{j=1}^{n} z_j \prod_{k \neq j} (z -z_k)^2 = \frac{m^2}{4}
(g_1(z)^2-g_2(z)^2).
\end{equation*}
Let $1\leq j \leq n$ be arbitrary. Substituting $z = z_j$ and taking square
roots yields
\[
\alpha \, z_j \prod_{k \neq j} (z_j -z_k) = \pm \frac{m}{2} \, g_2(z_j),
\]
which, by \eqref{g1z}, is equivalent to
\begin{equation}\label{dereq1}
z_j \, g_1'(z_j)=  \eps_j \frac{m}{2} \, g_2(z_j),
\end{equation}
where $\eps_j = \pm 1$.

\begin{lemma}
For all  $j$ and $k$, $\eps_j = \eps_k$.
\end{lemma}

\begin{proof}
First, for any $j$,
\begin{align*}
\arg g_1'(z_j) &= \lim_{\delta \rightarrow 0+} \left(\arg g_1(z_j e^{i \delta})
- \arg (z_j e^{i \delta} - z_j)\right)\\ &= \lim_{\delta \rightarrow 0+} \arg
g_1(z_j e^{i \delta}) - \arg z_j - \frac{\pi}{2}
\end{align*}
and therefore
\begin{equation}\label{argeq1}
\arg (z_j g_1'(z_j)) = \lim_{\delta \rightarrow 0+} \arg g_1(z_j e^{i \delta})
- \frac{\pi}{2}\,.
\end{equation}
Second, if $z \in T$ with $g_1(z) \neq 0$ and $g_2(z) \neq 0$, then by
\eqref{argument},
\[
\arg \frac {g_1(z)}{g_2(z)}  \equiv  \frac {\pi}{2} \pmod \pi.
\]
 Since $g_1(z)$ and $g_2(z)$ are polynomials with single
zeroes only, their arguments change continuously on $T$ apart from their
zeroes, where a jump of $\pi$ occurs. Observe that the zeroes of $g_1(z)$ and
$g_2(z)$ are alternating on $T$ (as the zeroes of $g_2$ are the local maximum
places of $h / h^*$). Hence,
\[
\lim_{\delta \rightarrow 0+} \arg \frac{g_1(z_j e^{i \delta})}{g_2(z_j e^{i
\delta})}
\]
is the same for every $j$. Now (\ref{argeq1}) yields that
\[ \arg \frac{z_j g_1'(z_j)}{g_2(z_j)}\]
does not depend on $j$ either, and by \eqref{dereq1}, the same is true for
$\eps_j$.
\end{proof}

Let $\eps_j = \eps= \pm 1$. From (\ref{dereq1}), we conclude that  the
polynomial
\[
z \, g_1'(z) - \eps \frac{m}{2} \, g_2(z)
\]
of degree $n$ attains 0 at all $(z_j)_1^n$, and hence its zeroes agree with
those of $g_1(z)$. Therefore there exists a complex number $\gamma$, such that
\[
z \, g_1'(z) - \eps \frac{m}{2} \, g_2(z)  = \gamma \, g_1(z),
\]
and thus
\begin{equation}\label{g1g1'}
 \eps \frac{m}{2} \, g_2(z)  = z \, g_1'(z) - \gamma \, g_1(z).
\end{equation}
Equating the leading coefficients, referring to (\ref{g1g2}), gives
\begin{equation}\label{mgamma}
\eps \frac{m}{2} \, \beta = (n - \gamma) \alpha,
\end{equation}
which, with the aid of (\ref{leadco}), yields that $\gamma \in \R$.

Finally, by comparing the leading coefficients and the constant terms in
(\ref{g1g1'}) and using the form (\ref{g1g2}), we deduce that $ (n-
\gamma)\alpha = \gamma \alpha $ and, since $\alpha \neq 0$,
\[
\gamma = \frac{n}{2}\,.
\]
Taking absolute values in (\ref{mgamma}) and referring to (\ref{leadco}), we
arrive at~$m = n$, which proves \eqref{planpolineq}.

Note that the proof gives more than the desired inequality: it shows that every
locally maximal set is actually a maximal set.

Next, we have to show that a set is locally extremal if and only if it is
equally distributed. First, assume that for some $c \in T$, $z_j^n = c$ for
every $j = 1,2,\dots,n$. Choosing $z_0$ to be the midpoint of the smaller arc
between two consecutive $z_j$'s, the sharpness of \eqref{planpolineq} follows
from setting $t = \pi / n$ in the identity
\begin{equation}\label{cosform}
 \sum_{j=1}^n \sin^{-2} \left(\frac{t}{2}-\frac{ j \pi}{n} \right)
  = \frac{2 n^2}{1-\cos n t}\;,
\end{equation}
which can be proved using Fej\'er kernels or Chebyshev polynomials; it also
follows from \eqref{e4} by setting
\[
Q(t) =  \sin  \frac{n t} 2  = (-1)^n\, 2^{n-1}\,\prod_{j=1}^n \sin
\left(\frac{t}{2}-\frac{ j \pi}{n}\right).
\]

 Finally, we prove that any locally extremal set is equally distributed,
based on an idea of L. Fejes T\'oth \cite{FT}. Let $(z_j)_1^n$ be a maximal set
with $z_j = e^{i t_j}$, and let $M^2(z_1, \dots, z_n)$ be attained at the
points $e^{i s_j}$, $j= 1, \dots, n$. Assume that $0\leq \theta_1< s_1< t_2 <
s_2 < \dots < t_n < s_n< 2 \pi$. Then, by Lemma~\ref{equimini},
\begin{equation}\label{sumjk}
\sum_{j=1}^n \sum_{k=1}^n \sin^{-2}\left(\frac{s_j - t_k} {2} \right) = n^3.
\end{equation}
For the sake of simplicity, the indices of the $t_j$'s will be understood
cyclically, i.e. $t_k = t_j$ for $j \equiv k \mod n$. Then, by Jensen's
inequality and \eqref{cosform},
\begin{align*}
 2 \sum_{k=1}^n &\sum_{j=1}^n \sin^{-2}\left( \frac{s_j -
t_k}{2}\right) \\&= \sum_{k=1}^n \sum_{j=1}^n \sin^{-2}\left( \frac{s_j -
t_{j+k}}{2} \right) +
\sin^{-2}\left(\frac{s_j - t_{j-k+1}}{2}\right)\\
&\geq  \sum_{k=1}^n 2 n \sin^{-2} \left( \frac{(2k -1)\pi}{2 n} \right) = 2
n^3.
\end{align*}
Thus, by \eqref{sumjk}, equality holds in Jensen's inequality at all instances.
Hence, the strict convexity of $\sin^{-2}(t/2)$ implies that $s_j- t_j= t_{j+1}
- s_j = \pi/(2n)$ for every $j$.
\end{proof}

We remark that starting from an arbitrary point set $(z_j)_1^n \subset T$,
defining $g_1(z)$ by \eqref{g1z}, and taking $m=n$ and $\gamma= n/2$, the
function $g_2(z)$ given by \eqref{g1g1'} has its zeroes where the modulus of
$g_1(z)$ is locally maximal. Thus, by \eqref{rateq3}, the proof implicitly
shows that in the extremal cases, $\sum |z-z_j|^{-2}$ and $\prod |z -
z_j|^{-1}$ have the same local minimum places on $T$.

\section{Second approach - Derivatives}

\begin{proof}[Second proof of the Theorem]
Associated with $z_j \in T$ we write $z_j = e^{it_j}$, $t_j \in [0,2\pi)$,
$j=1,2,\ldots, n$. We define
\begin{equation} \label{e3}
Q(t) = \prod_{j=1}^n{\sin \frac{t-t_j}{2}}\,.
\end{equation}
Then
\[
\frac{Q'(t)}{Q(t)} = \frac 12 \sum_{j=1}^n {\cot {\frac{t-t_j}{2}}}
\]
and
\begin{equation}\label{e4}
\begin{split}
\frac{Q''(t)Q(t)  - (Q'(t))^2}{(Q(t))^2} &= \left( \frac{Q'(t)}{Q(t)} \right)'\\
&\hspace{-0.5 cm}= -\frac 14 \sum_{j=1}^n{\csc^2{\frac {t-t_j}{2}}} = -\frac 14
\sum_{j=1}^n{\sin^{-2}{\frac {t-t_j}{2}}}\,.
\end{split}
\end{equation}
Observe that $Q$ and $Q'$ are entire functions of type $n/2$ (in fact they are
trigonometric polynomials of degree $n/2$ if $n$ is even), so by Bernstein's
inequality we have
\[
\max_{t \in {\Bbb R}}{|Q'(t)|} \leq \frac n2 \max_{t \in {\Bbb R}}{|Q(t)|}
\]
and
\begin{equation}\label{e5}
\max_{t \in {\Bbb R}}{|Q''(t)|} \leq \left( \frac n2 \right)^2 \max_{t \in
{\Bbb R}} {|Q(t)|}\,.
\end{equation}
Let $t_0 \in {\Bbb R}$ be chosen so that
\[
|Q(t_0)| = \max_{t \in \R}{|Q(t)|}\,.
\]
Then $Q'(t_0) = 0\,.$ Hence combining \eqref{e4} and \eqref{e5} we obtain
\[
\frac 14 \sum_{j=1}^n{\sin^{-2}{\frac {t_0-t_j}{2}}} =
\frac{|Q^{\prime\prime}(t_0)|}{|Q(t_0)|} \leq \frac{n^2}{4}\,.
\]
Introducing $z_0 := e^{it_0}$, we arrive at the desired inequality:
\[\sum_{j=1}^n{\frac{1}{|z_0-z_j|^2}} = \frac 14 \sum_{j=1}^n{\sin^{-2}{\frac {t_0-t_j}{2}}}
\leq \frac{n^2}{4}\,.
\]

Suppose now that the inequality \eqref{planpolineq} is sharp. Then equality
holds in Bernstein's inequality for $Q$, that is, $Q$ is of the form \eqref{e2}
with $\tau = n/2$. Here $a \neq 0$ otherwise $|Q(t)| = |b|$ identically for $t
\in {\Bbb R}$, a contradiction. Then the zeros of $Q(z)$ satisfy $e^{inz} =
-b/a$. Since the zeros of $Q$ are real, we have $|-b/a| = 1$, and each $z_j =
e^{it_j}$ satisfies the equation $z^n = -b/a$. Obviously the zeros $t_j$,
$j=1,2,\ldots, n$, of $Q$ on the period are distinct, hence $z_j = e^{it_j}$,
$j=1,2,\ldots, n$, are also distinct.

Now suppose that the numbers $z_j = e^{it_j}$ are distinct and there is a
number $c \in T$ such that $z_j^n = c$ for all $j=1,2,\dots,n$. Then $Q$ is of
the form \eqref{e2} with $\tau = n/2$. Choosing $z_0 = e^{it_0}$ to be the
midpoint of the smaller arc between two consecutive points $z_j$ on $T$, we
obtain $Q'(t_0) = 0$. Hence  \eqref{e4} implies that
\[
\frac 14 \sum_{j=1}^n{\sin^{-2}{\frac {t_0-t_j}{2}}} =
\sum_{j=1}^n{\frac{1}{|z_0-z_j|^2}} = \frac{|Q^{\prime\prime}(t_0)|}{|Q(t_0)|}
= \frac{n^2}{4}\,,
\]
where in the last equality we used that $Q$ is of the form \eqref{e2}.
\end{proof}

\section{Acknowledgements}
The authors thank I. B\'ar\'any for fruitful discussions,  K. Stolarsky for citing
various references, and Sz. Gy. R\'ev\'esz for the numerous suggestions and remarks
that improved the presentation of our results.

\end{document}